\documentclass[10pt]{article}
\usepackage{amssymb,amsmath}
\usepackage[utf8]{inputenc}
\usepackage[english]{babel}
\usepackage{amsmath, amssymb, amsfonts, amsthm,a4wide}
\usepackage{graphicx}
\usepackage{color}
\usepackage{bm}
\usepackage{subfigure}
\usepackage{graphicx}
\usepackage{mathabx}
\usepackage{multirow}
\usepackage{setspace}
\usepackage{titlesec}

\newtheorem{theorem}{Theorem}[section]
\newtheorem{lemma}[theorem]{Lemma}
\newtheorem{proposition}[theorem]{Proposition}
\newtheorem{corollary}[theorem]{Corollary}

\newtheorem{definition}{Definition}[section]
\newtheorem{remark}{Remark}[section]
\newtheorem{example}{Example}[section]

\titleformat{\paragraph}
{\normalfont\normalsize}{\theparagraph}{1em}{}
\titlespacing*{\paragraph}
{0pt}{3.25ex plus 1ex minus .2ex}{1.5ex plus .2ex}

\DeclareFontFamily{U}{MnSymbolC}{}
\DeclareSymbolFont{MnSyC}{U}{MnSymbolC}{m}{n}
\DeclareFontShape{U}{MnSymbolC}{m}{n}{
    <-6>  MnSymbolC5
   <6-7>  MnSymbolC6
   <7-8>  MnSymbolC7
   <8-9>  MnSymbolC8
   <9-10> MnSymbolC9
  <10-12> MnSymbolC10
  <12->   MnSymbolC12}{}
\DeclareMathSymbol{\lefthook}{\mathbin}{MnSyC}{'270}

\title{Lagrangian Engel Structures}
\footnotesize{\author{Zhiyong Zhao\\
Mathematics Department\\
Duke University, Durham, NC 27708-0230, USA\\
E-mail:  zhiyong.zhao@duke.edu}
}

\date{}
\begin{document}
\maketitle
\begin{center}{\bf Abstract} \end{center}
\begin{center}
\begin{minipage}{130mm}  
We study the geometry of Engel structures, which are 2-plane fields on 4-manifolds satisfying a generic condition, that are compatible with other geometric structures. A \em{Lagrangian} Engel structure is an Engel 2-plane field on a symplectic 4-manifold for which the 2-planes are Lagrangian with respect to the symplectic structure. We solve the equivalence problems for Lagrangian Engel structures and use the resulting structure equations to classify homogeneous Lagrangian Engel structures.  This allows us to determine all compact, homogeneous examples. Compact manifolds that support homogeneous Lagrangian Engel structures are diffeomorphic to quotients of one of a determined list of nilpotent or solvable 4-dimensional Lie groups by co-compact lattices.



\end{minipage}
\end{center}

\section{Introduction}
A \emph{distribution} is a subbundle $D\subset TM$ of the tangent bundle of a manifold $M$. We will consider certain distributions with special properties, for example, distributions with some integrability conditions.  We will study \emph{Engel structures}, which are certain non-integrable distributions defined on 4-manifolds. We will see that, locally, all Engel structures are isomorphic but the global theory of Engel structures is not trivial. A 4-manifold can carry many nonisomorphic Engel structures \cite{MR1350504}. There are relations between contact structures on 3-dimensional manfiolds and Engel structures. For example, V. Gershkovich \cite{MR1350504} proved that each Engel manifold carries a canonical one-dimensional foliation and an Engel structure defines a contact distribution on any three-dimensional submanifold transversal to the canonical foliation.

We will solve the equivalence problem for Lagrangian Engel structures and present the classification of compact quotients of homogeneous Lagrangian Engel structures. Engel structures (to be defined below) can be characterized in terms of the derived system construction.
 
 \begin{proposition} \cite{MR1083148}\label{pro:1}
 Given a Pfaffian system $I$, there exists a bundle map $\delta : I\rightarrow\Lambda^2(T^*M/I)$ that satisfies $\delta\omega \equiv d\omega\mod(I)$ for all $\omega\in\Gamma(I)$.
 \end{proposition}

 \begin{definition}
 By Proposition \ref{pro:1}, we have a bundle map $\delta$. Set $I^{(1)} = \ker \delta$ and call $I^{(1)}$ \emph{the first derived system}. 
  Continuing with this construction, we can get a filtration
  \[I^{(k)}\subset \cdots \subset I^{(2)}\subset I^{(1)}\subset I^{(0)} = I,    \]
  defined inductively by
  \[I^{(k+1)} = (I^{(k)})^{(1)}\, .\]
$I^{(k)}$ is called \emph{the $kth$ derived system}. 
 \end{definition}
 Now we present the definition and characterization of Engel structures.
 
\begin{definition}[Engel Structure]
Given a 4-manifold $M$ and a Pfaffian system $I\subset T^*M$, an \emph{Engel structure} is a sub-bundle $D = I^\perp$ of the tangent bundle of $M$ that satisfies: $(1)$ $I$ is of rank 2, $2$  $I^{(1)}$ is of rank 1 and $3$ $I^{(2)} = 0$.  A manifold endowed with an Engel structure $D = I^\perp$ is called an \emph{Engel manifold}. 
\end{definition} 

\begin{definition}
Given an ideal $\mathcal I$ generated by a Pffafian system $I$, a vector field $\xi$ is called a \emph{Cauchy characteristic vector field} of  $\mathcal I$ if $\xi\lefthook \mathcal I \subset \mathcal I$. At a point $x\in M$, the set of Cauchy characteristic vector fields is
\[A(I)_x = \{\xi_x\in T_xM | \xi_x\lefthook\mathcal I_x \subset\mathcal I_x\} \subset I^\bot\]
and the \emph{retracting space} or \emph{Cartan system} is defined to be
\[C(I)_x = A(I)_x^\bot \subset T_x^*M\, .\]
\end{definition}
By the definition of Engel structure $I^\bot$, there is a canonical flag of sub-bundles
\[0\subset I^{(1)}\subset I\subset C(I^{(1)}) \subset T^*M\, .\]
V. Gershkovich \cite{MR1350504} proved the following theorem which can also be found in \cite{MR1452539}.
\begin{theorem}
If an orientable 4-manifold admits an orientable Engel structure, then it has trivial tangent bundle.
\end{theorem}

T. Vogel \cite{MR2480602} proved the converse of the above theorem:
 \begin{theorem}
Every parallelizable 4-manifold admits an orientable Engel structure.
\end{theorem}
Thus for an orientable 4-manifold, parallelizability is equivalent to the existence of an orientable Engel structure. This is a global characterization of manifolds that support orientable Engel structures. Locally, we have the following \emph{Engel normal form} \cite{MR1083148}, which implies that there is no local invariant for Engel structures, i.e., all Engel structures are locally equivalent.
\begin{theorem}[Engel normal form] \label{the:1}
Let $I$ be a Pfaffian system on $M^4$ such that $I^\perp\subset TM$ is an Engel structure. Then every point of $M$ has an open neighborhood $U$ on which there exists local coordinates  $(x, y_0, y_1, y_2): U\rightarrow \mathbb{R}^4$ such that
 \[I|_{U} =\{dy_0-y_1 dx, dy_1 - y_2 dx\}\, . \]
\end{theorem}
In this paper, we will consider Lagrangian Engel structures.
\begin{definition}
A \emph{Lagrangian Engel structure} $(M, \Omega, D)$ is a 4-manifold $M$ endowed with \emph{a symplectic form $\Omega$ and an Engel 2-plane field $D$ that is Lagrangian for $\Omega$}. If we let $I = D^\perp\subset T^*M$ denote the annihilator, then $\Omega\in \langle I\rangle$.
\end{definition}

\section{Geometry of Lagrangian Engel Structures}
A coframing $\omega = (\omega_1, \omega_2, \omega_3, \omega_4)$ such that the symplectic structure can be written as
\[\Omega= \omega_1\wedge\omega_3 + \omega_2\wedge\omega_4\]
while $I = \langle\omega_1, \omega_2\rangle$ and $I^{(1)} = \langle \omega_1\rangle$ will be said to be \emph{0-adapted} to $(M, \Omega, D)$.

\begin{proposition}
The 0-adapted coframings are the sections of a $G$-structure on $M$ where $G\subset GL(4, \mathbb{R})$ is the 6-dimensional subgroup
\begin{equation}\label{eq:50}
G =\left \{  \left.\left[ {\begin{array}{cc}
                                    B_{11} & 0\\
                                   (B_{11}^T)^{-1} S & (B_{11}^T)^{-1}\\
                                    \end{array} } \right]
                                    \right | B_{11}= \left[ {\begin{array}{cc}
                                  b_{11}   & 0\\
                                   b_{21} & b_{22}\\
                                    \end{array}  }\right] \text{ and }
                                     S\in\mathbb{R}^{2\times 2}, S = S^T\right\}
                                     \end{equation}
\end{proposition}

\begin{proof}
Assume $(\tilde\omega_1, \tilde\omega_2, \tilde\omega_3, \tilde\omega_4)$ is a new coframing and the Engel structure in the new coframing is $I = \langle\tilde\omega_1, \tilde\omega_2\rangle$ and $I^{(1)} = \langle\tilde\omega_1\rangle$. According to the definition of a coframing on a manifold, there exists a matrix
\[ B =\left[ {\begin{array}{cc}
                                    B_{11} & B_{12}\\
                                   B_{21} & B_{22}\\
                                    \end{array} } \right] 
\]                                    
such that                                                                      
 \[\left[ {\begin{array}{c}
                                   \tilde\omega_1 \\
                                   \tilde\omega_2\\
                                   \tilde\omega_3\\
                                   \tilde\omega_4\\
                                    \end{array} } \right] = B \left[ {\begin{array}{c}
                                   \omega_1 \\
                                   \omega_2\\
                                   \omega_3\\
                                   \omega_4\\
                                    \end{array} } \right],\]
where $B_{11}, B_{12}, B_{21},B_{22}$ are $2\times 2$ matrices.
Since $I = \langle\tilde\omega_1, \tilde\omega_2\rangle = \langle\omega_1, \omega_2\rangle$ , the block $B_{12} = 0$. 

Let
\[J = \left[ {\begin{array}{cc}
                                    0 & I_2\\
                                   -I_2 & 0\\
                                    \end{array} } \right],\]
where $I_2$ is the identity matrix of dimension 2. To keep the symplectic structure invariant under the transformation, the matrices $B_{ij}$ satisfy
\[\left[ {\begin{array}{cc}
                                    B_{11} & 0\\
                                   B_{21} & B_{22}\\
                                    \end{array} } \right] ^TJ\left[ {\begin{array}{cc}
                                    B_{11} & 0\\
                                   B_{21} & B_{22}\\
                                    \end{array} } \right] = J.\]
Thus
\begin{align}\label{eq:22}
B_{11}^T B_{21} &=  B_{21}^T B_{11}\, ,\nonumber\\
B_{11}^T B_{22} &=  I_2\, .
\end{align}
Define $S = B_{11}^T B_{21} $. Then from (\ref{eq:22}), $S = S^T$. The element of the structure group can be written as
\[ B = \left[ {\begin{array}{cc}
                                    B_{11} & 0\\
                                   (B_{11}^T)^{-1} S & (B_{11}^T)^{-1}\\
                                    \end{array} } \right].\]

Since $I^{(1)} = \langle\omega_1\rangle = \langle\tilde\omega_1\rangle$, $B_{11}$ must be of the form $ \left[ {\begin{array}{cc}
                                  b_{11}   & 0\\
                                   b_{21} & b_{22}\\
                                    \end{array} } \right]$, where $ b_{11},  b_{21} , b_{22}$ can be any functions.
  
\noindent Therefore, the structure group is of the form (\ref{eq:50}).                             
\end{proof}

A Lagrangian Engel structure defines a $G$-structure, where $G$ is defined by (\ref{eq:50}).  We will prove that after reduction of the structure group, the manifold with a Lagrangian Engel structure belongs to at least one of the following categories:
\begin{enumerate}
\item  the manifold is not compact
\item there exists a canonical coframing for the Lagrangian Engel structure on the manifold
\end{enumerate}

Suppose $I = \langle\omega_1, \omega_2\rangle,\, I^{(1)} = \langle\omega_1\rangle$ and  $I^\perp$ is an Engel structure, then
\begin{align}\label{eq:23}
d\omega_1 &\not\equiv 0\ \ \ \ \ \ \ \mod\ \ \omega_1\, ,\nonumber\\
d\omega_1 &\equiv 0\ \ \ \ \ \ \ \mod\ \ \omega_1, \omega_2\, ,\nonumber\\
d\omega_2 &\not\equiv 0\ \ \ \ \ \ \ \mod\ \ \omega_1, \omega_2\, .
\end{align}

By (\ref{eq:23}), there exists a function $A\neq 0$ such that
\[d\omega_2 \equiv A\, \omega_3\wedge\omega_4\ \ \ \ \ \ \ \mod\ \ \omega_1, \omega_2\]
We can arrange $A = 1$ by dividing $\omega_2$ by $A$. Such coframings will be said to be \emph{1-adapted}. They are the sections of a $G_1$-structure, where $G_1\subset G$ is defined by
\begin{equation}\label{eq:31}
b_{11}b_{22}^2 = 1\, .
\end{equation}
Now $B_{11}$ is of the form $ \left[ {\begin{array}{cc}
                                  b_{22}^{-2}  & 0\\
                                   b_{21} & b_{22}\\
                                    \end{array} } \right]$
and
$B_{11}^{-1} =  \left[ {\begin{array}{cc}
                                  b_{22}^{2}  & 0\\
                                   -b_{22}b_{21} & b_{22}^{-1}\\
                                    \end{array} } \right]$.
After this arrangement,
\begin{equation}\label{eq:35}
d\omega_2 \equiv \omega_3\wedge\omega_4\ \ \ \ \ \ \ \mod\ \ \omega_1, \omega_2\, .
\end{equation}

By (\ref{eq:23}), there exist functions $p_3$ and $p_4$ such that
\begin{equation}\label{eq:100}
d\omega_1 \equiv (p_3\omega_3 + p_4\omega_4)\wedge\omega_2\ \ \ \ \ \ \ \mod\ \ \omega_1
\end{equation}
and at least one of $p_3$ and $p_4$ is nonzero. Since we will mainly focus on the classification of homogeneous Lagrangian Engel structures, we will study the cases where either $p_3\equiv 0$ or $p_3$ never vanishes.

Recall that the symplectic structure is $\Omega= \omega_1\wedge\omega_3 + \omega_2\wedge\omega_4$. By (\ref{eq:100}),
\begin{equation}\label{eq:59}
\omega_1\wedge d\omega_1\wedge\omega_4 =  \frac{p_3}{2}\,\Omega\wedge\Omega\, .
\end{equation}
If the coframing is changed under the structure group $G_1$, the function $p_3$ is changed to $b_{22}^{-5} p_3$. Thus $p_3$ is well-defined up to scaling by $b_{22}^{-5}$. By (\ref{eq:100}), the Cartan system is $C(\langle\omega_1\rangle) = \langle \omega_1, \omega_2, (p_3\omega_3 + p_4\omega_4)\rangle$ and the symplectic complement of $\langle\omega_1\rangle$ is $\langle\omega_1\rangle^\perp = \langle\omega_1, \omega_2, \omega_4\rangle$. Generally, $C(\langle\omega_1\rangle) \neq \langle\omega_1\rangle^\perp$.  If $C(\langle\omega_1\rangle) \neq \langle\omega_1\rangle^\perp$, this type of Lagrangian Engel structures is said to be 
\emph{generic}.  If $C(\langle\omega_1\rangle) = \langle\omega_1\rangle^\perp$, this type of Lagrangian Engel structures is said to be \emph{non-generic}. 
%

%
\subsection{Geometry of Lagrangian Engel Structures in Generic Case}
In the generic case, we have the following theorem:

\begin{theorem}[Lagrangian Engel Structures in Generic Case]\label{th:LG}
Given a symplectic manifold $(M, \Omega, I)$ with a symplectic structure $\Omega$ and an Engel structure $D = I^\perp$. On the domain where $p_3 \neq 0$ in equation (\ref{eq:59}), there exists a unique $0$-adapted coframing $\omega = (\omega_1, \omega_2, \omega_3, \omega_4)$ satisfying
\begin{align*}
d\omega_1 &= \omega_3\wedge\omega_2 + (a\omega_3 + b\omega_4)\wedge\omega_1\, ,\\
d\omega_2 &=(c\omega_2 + e\omega_3 + f\omega_4)\wedge\omega_1 +  \omega_3\wedge\omega_4\, ,
\end{align*}
where $a,b,c,e,f$ are functions on $M$.
\end{theorem}

\begin{proof}
Under the transformation of the structure group, the structure equation is transformed to
\begin{equation}\label{eq:49}
d\omega_1 \equiv\left(p_3b_{22}^5\omega_3 + b_{22}^2(-p_3b_{22}^2b_{21} + p_4)\omega_4\right)\wedge\omega_2\ \ \ \ \ \ \ \mod\ \ \omega_1\, .
\end{equation}
By scaling $\omega_1$ via $b_{22}^5$, we can arrange $p_3 = 1$. This fixes $b_{22} = 1$.  By adding a multiple of $\omega_4$ to $\omega_3$, we can arrange $p_4 = 0$. This fixes $b_{21} = 0$. The structure equation is
\begin{equation}\label{eq:51}
d\omega_1\equiv \omega_3\wedge\omega_2\ \ \ \ \ \ \ \mod\ \ \omega_1\, .
\end{equation}
The element of the structure group reduces to the following form
\[\left[ {\begin{array}{cccc}
                                    1 & 0 & 0& 0\\
                                    0 & 1 & 0& 0\\
                                    S_{11}  &S_{12} & 1 & 0\\
                                    S_{12} & S_{22} & 0 & 1
                                    \end{array} } \right]\, . \
\]
Recall that $d\omega_2\equiv \omega_3\wedge\omega_4\ \mod\omega_1,\omega_2$. Thus there exist functions $v_3$ and $v_4$ such that
\begin{equation}
d\omega_2\equiv (v_3\omega_3 + v_4\omega_4)\wedge\omega_2 +  \omega_3\wedge\omega_4\ \ \ \ \ \ \ \mod\ \ \omega_1.
\end{equation}

By adding a multiple of $\omega_2$ to $\omega_3$, we can arrange $v_4 = 0$. This fixes $S_{12} = 0$. By adding a multiple of $\omega_2$ to $\omega_4$, we can arrange $v_3 = 0$. This fixes $S_{22} = 0$. 
Thus
\begin{equation}
d\omega_2\equiv \omega_3\wedge\omega_4\ \ \ \ \ \ \ \mod\ \ \omega_1.
\end{equation}
From equation (\ref{eq:51}), there exist functions $u_2,u_3, u_4$ such that
\begin{equation}
d\omega_1= \omega_3\wedge\omega_2 + (u_2\omega_2 + u_3\omega_3 + u_4\omega_4)\wedge\omega_1\, .
\end{equation}
By adding a multiple of $\omega_1$ to $\omega_3$, we can arrange $u_2 = 0$. This yields $S_{11} = 0$. Now the structure group of the coframing contains only the identity element. We get an $e$-structure. The structure equation is
\begin{equation}
\begin{aligned}
d\omega_1 &= \omega_3\wedge\omega_2 + (u_3\omega_3 + u_4\omega_4)\wedge\omega_1\, ,\\
d\omega_2 &\equiv \omega_3\wedge\omega_4\ \ \ \ \ \ \ \mod\ \ \omega_1\, .
\end{aligned}
\end{equation}
\end{proof}
By a Theorem of Kobayashi \cite{MR0088766},
\begin{corollary}
In generic case, the symmetry group of a Lagrangian Engel structure acts freely on the underlying connected manifold.
\end{corollary}

\subsection{Geometry of Lagrangian Engel Structures in Non-Generic Case}

Now we will study the geometry of Lagrangian Engel structures in non-generic case.
Since $p_3 \equiv 0$ and at least one of $p_3$ and $p_4$ is nonzero, then $p_4$ never vanishes.
We can arrange $p_4 = \pm 1$ via dividing $\omega_1$ by $|p_4|$. Then the entry $b_{22}$ is fixed to be $\pm1$. From the expression of the symplectic structure, there is a transformation 
\[\omega_2\rightarrow -\omega_2, \ \ \ \ \ \omega_4 \rightarrow -\omega_4\, .\]
that fixes the symplectic structure and Engel structure. We can fix $b_{22} = 1$ by this transformation.

Now 
\begin{equation}\label{eq:37}
d\omega_1 \equiv \pm\omega_4\wedge \omega_2\ \ \ \ \ \ \ \mod\ \ \omega_1\, .
\end{equation}
$\omega_1$ is uniquely defined by equation (\ref{eq:37}).

By (\ref{eq:37}), there exist functions $A_2, A_3$ and $A_4$ such that
\begin{equation}\label{eq:38}
d\omega_1 = \pm \omega_4\wedge \omega_2 + (A_2\omega_2 + A_3\omega_3 + A_4\omega_4)\wedge\omega_1 \, .
\end{equation}
Under a change of adapted coframing, a new coframing $\tilde{\omega}$ satisfies
\begin{equation}\label{eq:33}
\begin{aligned}
d\tilde\omega_1 &= \pm\tilde\omega_4\wedge\tilde\omega_2 + A_3\tilde\omega_3\wedge\tilde\omega_1 + (\pm b_{21} - b_{21}A_3 + A_4)\tilde\omega_4\wedge\tilde\omega_1\\
&+ \left(\mp S_{12} \pm b_{21}S_{22} + A_2 + A_3(S_{12} - b_{21} S_{22}) + A_4S_{22}\right)\tilde\omega_2\wedge\tilde\omega_1\, .
\end{aligned}
\end{equation}
By comparing (\ref{eq:38}) and (\ref{eq:33}), $A_3$ is an invariant of Lagrangian Engel structures in the case $p_3\equiv 0$. And note that by (\ref{eq:38}),
\begin{equation}\label{eq:60}
d\omega_1\wedge\omega_2\wedge\omega_4 =  -\tfrac{A_3}{2}\,\Omega\wedge\Omega\, .
\end{equation}

If $A_3 \equiv \pm 1$,  (\ref{eq:33}) is equivalent to 
\begin{equation}\label{eq:101}
d\tilde\omega_1 = \pm\tilde\omega_4\wedge\tilde\omega_2 \pm\tilde\omega_3\wedge\tilde\omega_1 +  A_4\tilde\omega_4\wedge\tilde\omega_1+ \left( A_2 + A_4S_{22}\right)\tilde\omega_2\wedge\tilde\omega_1\, .
\end{equation}
Thus by comparing (\ref{eq:38}) and (\ref{eq:101}), $A_4$ is an invariant of Lagrangian Engel structures.

Based on the invariants $A_3$ and $A_4$ in (\ref{eq:38}), we will prove the following theorem: 

\begin{theorem}[Lagrangian Engel Structures in Non-Generic Case]\label{th:LNG}$       $
\begin{enumerate}
\item On the domain where $(A_3 \equiv 0)$ or $(A_3 \neq \pm 1$ and $A_3 \neq 0)$ in (\ref{eq:60}), we get an $e$-structure that whose defining conditions are
\begin{align*}
d\omega_1 &= \pm\omega_4\wedge\omega_2  + A_3\omega_3\wedge\omega_1\, ,\\
d\omega_2 &= (a\omega_1 + b\omega_4)\wedge\omega_2 +  c\omega_3\wedge\omega_1 + \omega_3\wedge\omega_4\, ,
\end{align*}
where $a,b,c$ are functions on $M$.
\item On the domain where $A_3 \equiv \pm 1$ in equation (\ref{eq:60}), there are two cases depending on whether $A_4$ is 0:
\begin{enumerate}
\item On the domain where $A_4 \equiv 0$, the structure equation is
\begin{equation}
\begin{aligned}
d\omega_1 &= \pm\omega_4\wedge\omega_2 \pm\omega_3\wedge\omega_1 + A_2\,\omega_2\wedge\omega_1\, ,\\
d\omega_2 &\equiv \omega_3\wedge\omega_4 \ \ \ \ \ \ \ \mod\ \ \omega_1\, .
\end{aligned}
\end{equation}
\item On the domain where $A_4 \neq 0$, the structure equation is
\begin{equation}
\begin{aligned}
d\omega_1 &= \pm\omega_4\wedge\omega_2 \pm\omega_3\wedge\omega_1 +  A_4\, \omega_4\wedge\omega_1\, ,\\
d\omega_2 &\equiv q_3\omega_3\wedge\omega_2 + \omega_3\wedge\omega_4\, . \ \ \ \ \ \ \ \mod\ \ \omega_1
\end{aligned}
\end{equation}
\end{enumerate}

\end{enumerate}

\end{theorem}

\begin{proof}

We consider the following 3 sub-cases:
\begin{enumerate}
\item $A_3\equiv 0$ 
\item $A_3 \equiv \pm 1$
\item  $A_3\neq 0$ and $A_3\neq \pm1$ 
 \end{enumerate}

\subsubsection{case $A_3\equiv 0$}
If $A_3\equiv 0$, (\ref{eq:33}) is equivalent to
\begin{equation}\label{eq:34}
d\tilde\omega_1 = \pm\tilde\omega_4\wedge\tilde\omega_2 + (\pm b_{21} + A_4)\tilde\omega_4\wedge\tilde\omega_1+ \left(\mp S_{12} \pm b_{21}S_{22} + A_2 + A_4S_{22}\right)\tilde\omega_2\wedge\tilde\omega_1\, .
\end{equation}

\noindent From the second term of (\ref{eq:34}), by adding a multiple of $\omega_1$ into $\omega_2$, we can arrange $A_4 = 0$. This yields $b_{21} = 0$. Now 
\[B_{11}=  \left[ {\begin{array}{cc}
                                  1  & 0\\
                                   0 & 1\\
                                    \end{array} } \right]= I_2\, .\]
Equation (\ref{eq:34}) is equivalent to
\begin{equation}
d\tilde\omega_1 = \pm\tilde\omega_4\wedge\tilde\omega_2 + \left(\mp S_{12} + A_2 \right)\tilde\omega_2\wedge\tilde\omega_1\, .
\end{equation}
\noindent From the second term of the right side of the above equation, by adding a multiple of $\omega_1$ into $\omega_4$, we can arrange $A_2 = 0$. This yields $S_{12} = 0$. Thus
\begin{equation}
d\omega_1 = \pm\omega_4\wedge\omega_2\, .
\end{equation}
 And
\[B =  \left[ {\begin{array}{cccc}
                                 1 & 0 & 0 & 0\\
                                 0 & 1 & 0 & 0\\
                                  S_{11}  & 0 & 1 & 0\\
                                   0 & S_{22} & 0 & 1\\
                                    \end{array} } \right]\, .\]
By (\ref{eq:35}), there exist functions $q_3$ and $q_4$ such that
\[d\omega_2 \equiv (q_3\omega_3 + q_4\omega_4)\wedge\omega_1 + \omega_3\wedge\omega_4 \ \ \ \ \ \ \ \mod\ \ \omega_2\, . \]
\noindent By adding a multiple of $\omega_1$ into $\omega_3$, we can arrange $q_4 = 0$. This yields $S_{11} = 0$. Thus
\begin{equation}
d\omega_2 \equiv q_3\omega_3\wedge\omega_1 + \omega_3\wedge\omega_4 \ \ \ \ \ \ \ \mod\ \ \omega_2\, .
\end{equation}
There exist functions $r_1, r_3$ and $r_4$ such that
\begin{equation}
d\omega_2 = (r_1\omega_1 + r_3\omega_3 + r_4\omega_4)\wedge\omega_2 +  q_3\omega_3\wedge\omega_1 + \omega_3\wedge\omega_4\, .
\end{equation}
\noindent By adding a multiple of $\omega_2$ into $\omega_4$, we can arrange $r_3 = 0$. This yields $S_{22} = 0$. 

Now the structure group contains only the identity element, i.e., we have found an $e$-structure. In this case, the structure equation is
\begin{equation}
\begin{aligned}
d\omega_1 &= \pm\omega_4\wedge\omega_2\, ,\\
d\omega_2 &= (r_1\omega_1 + r_4\omega_4)\wedge\omega_2 +  q_3\omega_3\wedge\omega_1 + \omega_3\wedge\omega_4\, .
\end{aligned}
\end{equation}

\subsubsection{case $A_3\equiv \pm1$}

(\ref{eq:33}) is equivalent to 
\begin{equation}\label{eq:46}
d\tilde\omega_1 = \pm\tilde\omega_4\wedge\tilde\omega_2 \pm\tilde\omega_3\wedge\tilde\omega_1 +  A_4\tilde\omega_4\wedge\tilde\omega_1+ \left( A_2 + A_4S_{22}\right)\tilde\omega_2\wedge\tilde\omega_1,
\end{equation}
where $A_4$ is an invariant of Lagrangian Engel structures in the case $A_3\equiv \pm1$. 

We will consider the following 2 sub-cases:
\begin{enumerate}
\item $A_4 \equiv 0$
\item $A_4 \neq 0$
\end{enumerate}

\textbf{ case : $A_4 \equiv 0$}
\begin{equation}\label{eq:52}
d\omega_1 = \pm\omega_4\wedge\omega_2 \pm\omega_3\wedge\omega_1 + A_2\omega_2\wedge\omega_1\, .
\end{equation}
Now $A_2$ is an invariant of Lagrangian Engel structures in this case. And
\[B =  \left[ {\begin{array}{cccc}
                                 1 & 0 & 0 & 0\\
                                 b_{21} & 1 & 0 & 0\\
                                  S_{11} - b_{21}S_{12}  & S_{12} - b_{21}S_{22} & 1 & -b_{21}\\
                                   S_{12} & S_{22} & 0 & 1\\
                                    \end{array} } \right]\]
By (\ref{eq:35}), there exist functions $q_3$ and $q_4$ such that
\[d\omega_2 \equiv (q_3\omega_3 + q_4\omega_4)\wedge\omega_2 + \omega_3\wedge\omega_4 \ \ \ \ \ \ \ \mod\ \ \omega_1\, . \]
\noindent By adding a multiple of $\omega_2$ to $\omega_4$, we can arrange $q_3 = 0$. This yields $S_{22} = 0$. By adding a multiple of $\omega_2$ to $\omega_3$, we can arrange $q_4 = 0$. This yields $S_{12}  \pm b_{21}= 0$. And
\begin{equation}\label{eq:53}
d\omega_2 \equiv \omega_3\wedge\omega_4 \ \ \ \ \ \ \ \mod\ \ \omega_1\, .
\end{equation}
There exist functions $r_2, r_3$ and $r_4$ such that
\[d\omega_2 = (r_2\omega_2 + r_3\omega_3 + r_4\omega_4)\wedge\omega_1 + \omega_3\wedge\omega_4\, .\]
The elements of the structure group are of the form
\[B =  \left[ {\begin{array}{cccc}
                                 1 & 0 & 0 & 0\\
                                 b_{21} & 1 & 0 & 0\\
                                  S_{11} \pm b_{21}^2  & \mp b_{21} & 1 & -b_{21}\\
                                   \mp b_{21} & 0 & 0 & 1\\
                                    \end{array} } \right]\, .\]
In this case, the structure group does not reduce to the trivial group. The structure group can be further reduced by considering the derivative of $\omega_3$ and $\omega_4$.

Now $\omega_1$ is uniquely defined by (\ref{eq:52}) and $\omega_2$ is uniquely defined up to an addition of a multiple of $\omega_1$ by (\ref{eq:53}). Thus $\omega_1\wedge\omega_2$ is uniquely defined by (\ref{eq:52}) and (\ref{eq:53}). Therefore,
\begin{equation}\label{eq:54}
\omega_1\wedge\omega_2\wedge d\omega_2 = \omega_1\wedge\omega_2\wedge\omega_3\wedge\omega_4
\end{equation}
is uniquely defined.

\textbf{ case : $A_4\neq 0$}

By (\ref{eq:46}), after adding a multiple of $\omega_2$ into $\omega_4$, we can arrange $A_2 = 0$. This yields $S_{22} = 0$. And
\begin{equation}\label{eq:55}
d\omega_1 = \pm\omega_4\wedge\omega_2 \pm\omega_3\wedge\omega_1 +  A_4\omega_4\wedge\omega_1\, .
\end{equation}

By (\ref{eq:35}), there exist functions $q_3$ and $q_4$ such that
\[d\omega_2 \equiv (q_3\omega_3 + q_4\omega_4)\wedge\omega_2 + \omega_3\wedge\omega_4 \ \ \ \ \ \ \ \mod\ \ \omega_1\, . \]
From this structure equation, $q_3$ is an invariant. By adding a multiple of $\omega_2$ and $\omega_4$ to $\omega_3$, we can arrange $q_4 = 0$. This yields $S_{12} + b_{21}q_3 \pm b_{21} = 0$. 

\begin{equation}\label{eq:56}
d\omega_2 \equiv q_3\omega_3\wedge\omega_2 + \omega_3\wedge\omega_4 \ \ \ \ \ \ \ \mod\ \ \omega_1\, .
\end{equation}

 Now $\omega_1$ is uniquely defined by (\ref{eq:55}) and $\omega_2$ is uniquely defined up to an addition of a multiple of $\omega_1$ by (\ref{eq:56}). Thus $\omega_1\wedge\omega_2$ is uniquely defined by (\ref{eq:55}) and (\ref{eq:56}). Therefore,
\begin{equation}\label{eq:57}
\omega_1\wedge\omega_2\wedge d\omega_2 = \omega_1\wedge\omega_2\wedge\omega_3\wedge\omega_4
\end{equation}
is uniquely defined.

%

\subsubsection{case $A_3\neq \pm 1$ and $A_3\neq 0$}
By (\ref{eq:33}), after adding a multiple of $\omega_4$ to $\omega_3$, we can arrange $A_4 = 0$. This yields $b_{21} = 0$. Then
\begin{equation}
d\omega_1 = \pm\omega_4\wedge\omega_2 + A_3\omega_3\wedge\omega_1+ \left(\mp S_{12} + A_2 + A_3S_{12}\right)\omega_2\wedge\omega_1.
\end{equation}
By adding a multiple of $\omega_2$ to $\omega_3$, we can arrange $A_2 = 0$.  This yields $S_{12} = 0$. Then
\begin{equation}
d\omega_1 = \pm\omega_4\wedge\omega_2 + A_3\omega_3\wedge\omega_1.
\end{equation}
By (\ref{eq:35}), there exist functions $q_3$ and $q_4$ such that
\[d\omega_2 \equiv (q_3\omega_3 + q_4\omega_4)\wedge\omega_1 + \omega_3\wedge\omega_4 \ \ \ \ \ \ \ \mod\ \ \omega_2. \]
\noindent By adding a multiple of $\omega_1$ to $\omega_3$, we can arrange $q_4 = 0$. This yields $S_{11} = 0$. Thus
\begin{equation}
d\omega_2 \equiv q_3\omega_3\wedge\omega_1 + \omega_3\wedge\omega_4 \ \ \ \ \ \ \ \mod\ \ \omega_2.
\end{equation}
There exist functions $r_1, r_3$ and $r_4$ such that
\begin{equation}
d\omega_2 = (r_1\omega_1 + r_3\omega_3 + r_4\omega_4)\wedge\omega_2 +  q_3\omega_3\wedge\omega_1 + \omega_3\wedge\omega_4.
\end{equation}
\noindent By adding a multiple of $\omega_2$ to $\omega_4$, we can arrange $r_3 = 0$. This yields $S_{22} = 0$. 

Now the structure group contains only the identity element, i.e., we have found an $e$-structure. In this case, the structure equation is
\begin{equation}
\begin{aligned}
d\omega_1 &= \pm\omega_4\wedge\omega_2 + A_3\omega_3\wedge\omega_1\, ,\\
d\omega_2 &= (r_1\omega_1 + r_4\omega_4)\wedge\omega_2 +  q_3\omega_3\wedge\omega_1 + \omega_3\wedge\omega_4\, .
\end{aligned}
\end{equation}

\end{proof}

\section{Classification of Homogeneous Lagrangian Engel Structures}
In this section, we derive the structure equation of homogeneous Lagrangian Engel structures via equivalence method \cite{MR1062197}. 
\begin{theorem}[Classification of Homogeneous Lagrangian Engel Structures]\label{th:hle}
There are at most 6 distinct families of homogeneous Lagrangian Engel structures that can \emph{have compact quotient manifolds}. These 6 families are listed as follows:
\begin{enumerate}
\item Case 1: 
 \[d\left[ {\begin{array}{c}
                                   \omega_1 \\
                                   \omega_2\\
                                   \omega_3\\
                                   \omega_4\\
                                    \end{array} } \right] =\left[ {\begin{array}{c}
                                   \omega_2\wedge\omega_3 + a\omega_1\wedge\omega_3 \\
                                 \omega_3\wedge\omega_4\\
                                 0\\
                                b\omega_2\wedge\omega_3\\
                                    \end{array} } \right]\]

\item Case 2:
 \[d\left[ {\begin{array}{c}
                                   \omega_1 \\
                                   \omega_2\\
                                   \omega_3\\
                                   \omega_4\\
                                    \end{array} } \right] =\left[ {\begin{array}{c}
                                   \omega_2\wedge\omega_3 + a\omega_1\wedge\omega_3 + b\omega_1\wedge\omega_4\\
                                b\omega_1\wedge\omega_3 + \omega_3\wedge\omega_4 + b\omega_2\wedge\omega_4\\
                                 0\\
                               0\\
                                    \end{array} } \right]\]

\item Case 3:
 \[d\left[ {\begin{array}{c}
                                   \omega_1 \\
                                   \omega_2\\
                                   \omega_3\\
                                   \omega_4\\
                                    \end{array} } \right] =\left[ {\begin{array}{c}
                                   \omega_2\wedge\omega_3 + a\omega_1\wedge\omega_3 - \frac{1}{4} a^2\omega_1\wedge\omega_4\\
                                -\frac{1}{4}a^2\omega_1\wedge\omega_3 + \omega_3\wedge\omega_4  -\frac{1}{4}a^2\omega_2\wedge\omega_4\\
                                 \frac{1}{2}a^2b(\omega_1\wedge\omega_3 - \omega_2\wedge\omega_4) + ab\omega_2\wedge\omega_3 - \frac{1}{4}a^3b\omega_1\wedge\omega_4\\
                               ab(\omega_1\wedge\omega_3 - \omega_2\wedge\omega_4) + 2b \omega_2\wedge\omega_3 - \frac{1}{2}a^2b\omega_1\wedge\omega_4\\
                                    \end{array} } \right]\]

\item Case 4:
 \[d\left[ {\begin{array}{c}
                                   \omega_1 \\
                                   \omega_2\\
                                   \omega_3\\
                                   \omega_4\\
                                    \end{array} } \right] =\left[ {\begin{array}{c}
                                   \omega_2\wedge\omega_3 + b\omega_1\wedge\omega_3 + a \omega_1\wedge\omega_4\\
                               (a^2 + \frac{ab^2}{4})\omega_1\wedge\omega_2 + a\omega_1\wedge\omega_3 + \omega_3\wedge\omega_4 + a\omega_2\omega_4 \\
                                (2a^3 + \frac{1}{2}a^2b^2)\omega_{12} + \frac{ab^2}{2}\omega_{13} + b\omega_{23} + a^2b\omega_{14} + 2a^2\omega_{24}\\
                               ab(-a - \frac{1}{4})\omega_{12} + ab\omega_{13} + (a - \frac{1}{4})\omega_{23} + (a^2 - \frac{ab^2}{4})\omega_{14} -ab\omega_{24}\\
                                    \end{array} } \right]\]

\item Case 5:
 \[d\left[ {\begin{array}{c}
                                   \omega_1 \\
                                   \omega_2\\
                                   \omega_3\\
                                   \omega_4\\
                                    \end{array} } \right] =\left[ {\begin{array}{c}
                                   \omega_1\wedge\omega_3 + \omega_2\wedge\omega_4\\
                                  a\omega_1\wedge\omega_2 + \omega_3\wedge\omega_4\\
                                  a( \omega_1\wedge\omega_3 + \omega_2\wedge\omega_4)\\
                                -a  \omega_1\wedge\omega_2 - \omega_3\wedge\omega_4\\
                                    \end{array} } \right]\]

\item Case 6:
 \[d\left[ {\begin{array}{c}
                                   \omega_1 \\
                                   \omega_2\\
                                   \omega_3\\
                                   \omega_4\\
                                    \end{array} } \right] =\left[ {\begin{array}{c}
                                   \omega_1\wedge\omega_3 + \omega_2\wedge\omega_4\\
                                  \omega_3\wedge\omega_4\\
                                  0\\
                                 a\omega_2\wedge\omega_3 - \omega_3\wedge\omega_4\\
                                    \end{array} } \right]\]
\end{enumerate}
\noindent where $a$ and $b$ are constants.
\end{theorem}

\begin{proof}

\noindent  Since the structure group of Lagrangian Engel structures is of the form (\ref{eq:50}), there exists Lie algebra-valued differential form 
\[\pi = \left[ {\begin{array}{cccc}
                                    \pi_1 & 0 & 0& 0\\
                                    \pi_2 & \pi_3 & 0& 0\\
                                    \pi_4 & \pi_5 & -\pi_1 & -\pi_2\\
                                     \pi_5 & \pi_6 & 0 & -\pi_3
                                    \end{array} } \right]  = (\pi_{ij})\]
such that the structure equation can be written as
\begin{equation}
d\omega_i = \sum \pi_{ij}\wedge\omega_j + \frac{1}{2}\sum\gamma_{ijk}\omega_j\wedge\omega_k\, ,
\end{equation}
where $\gamma_{ijk}$ are torsion terms. By (\ref{eq:23}), there exist functions $a_0, a_3, a_4$ such that
\[d\left[ {\begin{array}{c}
                                   \omega_1 \\
                                   \omega_2\\
                                    \end{array} } \right] =-\left[ {\begin{array}{cccc}
                                    \pi_1 & 0 & 0& 0\\
                                    \pi_2 & \pi_3 & 0& 0\\
                                    \end{array} } \right] \left[ {\begin{array}{c}
                                   \omega_1 \\
                                   \omega_2\\
                                   \omega_3\\
                                   \omega_4\\
                                    \end{array} } \right] + \left[ {\begin{array}{c}
                                   \omega_2\wedge(a_3 \omega_3 + a_4 \omega_4) \\
                                   a_0\omega_3\wedge\omega_4\\
                                    \end{array} } \right]\]
and at least one of $a_3$ and $a_4$ is nonzero and  $a_0 \neq 0$.
                                
By modifying the Lie algebra valued 1-forms $\pi_4, \pi_5, \pi_6$ and absorption of torsions, there exist functions $S_1$ and $S_2$ and 1-form $\tau$ such that                                    
 \[d\left[ {\begin{array}{c}
                                   \omega_3\\
                                   \omega_4\\
                                    \end{array} } \right] =-\left[ {\begin{array}{cccc}
                                    \pi_4 & \pi_5 & -\pi_1 & -\pi_2\\
                                     \pi_5 & \pi_6 & 0 & -\pi_3
                                    \end{array} } \right] \left[ {\begin{array}{c}
                                   \omega_1 \\
                                   \omega_2\\
                                   \omega_3\\
                                   \omega_4\\
                                    \end{array} } \right] + \left[ {\begin{array}{c}
                                  S_1 \omega_3\wedge\omega_4 + \tau\wedge\omega_2\\
                                  S_2 \omega_3\wedge\omega_4 - \tau\wedge\omega_1\\
                                    \end{array} } \right]\, ,\]
where $\tau = t_3\omega_3 + t_4\omega_4$ for some functions $t_3$ and $t_4$ . Thus after this absorption of torsions, we get $0$-adapted coframing such that
\begin{equation}\label{eq:61}
d\left[ {\begin{array}{c}
                                   \omega_1 \\
                                   \omega_2\\
                                   \omega_3\\
                                   \omega_4\\
                                    \end{array} } \right] =-\left[ {\begin{array}{cccc}
                                    \pi_1 & 0 & 0& 0\\
                                    \pi_2 & \pi_3 & 0& 0\\
                                    \pi_4 & \pi_5 & -\pi_1 & -\pi_2\\
                                     \pi_5 & \pi_6 & 0 & -\pi_3
                                    \end{array} } \right] \left[ {\begin{array}{c}
                                   \omega_1 \\
                                   \omega_2\\
                                   \omega_3\\
                                   \omega_4\\
                                    \end{array} } \right] + \left[ {\begin{array}{c}
                                   \omega_2\wedge(a_3 \omega_3 + a_4 \omega_4) \\
                                   a_0\omega_3\wedge\omega_4\\
                                  S_1 \omega_3\wedge\omega_4 + \tau\wedge\omega_2\\
                                  S_2 \omega_3\wedge\omega_4 - \tau\wedge\omega_1\\
                                    \end{array} } \right].
                                    \end{equation}                                
                                    
Since $\Omega$ is a symplectic form, $d\Omega = 0$. By (\ref{eq:61}) and $d\Omega = 0$, we get $S_2 = - a_4,\, S_1 = 0,\, t_3 = 0,\, t_4 = 0$. Now (\ref{eq:61}) is transformed into
 \begin{equation}\label{eq:62}
 d\left[ {\begin{array}{c}
                                   \omega_1 \\
                                   \omega_2\\
                                   \omega_3\\
                                   \omega_4\\
                                    \end{array} } \right] =-\left[ {\begin{array}{cccc}
                                    \pi_1 & 0 & 0& 0\\
                                    \pi_2 & \pi_3 & 0& 0\\
                                    \pi_4 & \pi_5 & -\pi_1 & -\pi_2\\
                                     \pi_5 & \pi_6 & 0 & -\pi_3
                                    \end{array} } \right] \left[ {\begin{array}{c}
                                   \omega_1 \\
                                   \omega_2\\
                                   \omega_3\\
                                   \omega_4\\
                                    \end{array} } \right] + \left[ {\begin{array}{c}
                                   \omega_2\wedge(a_3 \omega_3 + a_4 \omega_4) \\
                                   a_0\omega_3\wedge\omega_4\\
                                  0\\
                                  -a_4 \omega_3\wedge\omega_4\\
                                    \end{array} } \right].
                                    \end{equation}

\noindent Now we calculate the reduction of the group using the equivalence method. Calculate $d^2\omega_1 = 0$ by (\ref{eq:62})
\begin{align}\label{eq:24}
da_4 + a_3\pi_2 + a_4 \pi_1 &\equiv 0\ \ \ \ \ \ \ \ \   \mod\ \ \omega_1, \omega_2, \omega_3,\omega_4\, ,\nonumber\\
da_3 + a_3(2\pi_1 - \pi_3) & \equiv 0\ \ \ \ \ \ \ \ \   \mod\ \ \omega_1, \omega_2, \omega_3,\omega_4\, .
\end{align}
Since at least one of $a_3$ and $a_4$ is nonzero, there are two cases: $a_3 \neq 0$ or $a_3 = 0$.

\subsection{case $a_3 \neq 0$}
We can scale $a_3 = 1$ and translate $a_4 = 0$. Then from (\ref{eq:24})
\begin{align}\label{eq:63}
\pi_2 &\equiv 0\ \ \ \ \ \ \ \ \   \mod\ \ \omega_1, \omega_2, \omega_3,\omega_4\, ,\nonumber\\
2\pi_1 - \pi_3 & \equiv 0\ \ \ \ \ \ \ \ \   \mod\ \ \omega_1, \omega_2, \omega_3,\omega_4
\end{align}
This means $\pi_2,\,  2\pi_1 - \pi_3$ are basic.

\noindent Calculating $d^2\omega_2 = 0$ from (\ref{eq:62}) yields
\[da_0 + a_0(\pi_1 + 2 \pi_3)  \equiv 0\ \ \ \ \ \ \ \ \   \mod\ \ \omega_1, \omega_2, \omega_3,\omega_4\, .\]
Since $a_0 \neq 0$, we can scale $a_0 = 1$. Then
\begin{equation}\label{eq:64}
\pi_1 +2 \pi_3  \equiv 0\ \ \ \ \ \ \ \ \   \mod\ \ \omega_1, \omega_2, \omega_3,\omega_4\, .
\end{equation}
Thus from (\ref{eq:63}) and (\ref{eq:64}), we have
\[\pi_1 \equiv\pi_2\equiv  \pi_3  \equiv 0\ \ \ \ \ \ \ \ \   \mod\ \ \omega_1, \omega_2, \omega_3,\omega_4\, .\]
Define $\pi_i = \sum_{j=1}^{4}a_{ij}\omega_j$, where $i = 1,2,3$. Calculate $d^2\omega_4 = 0$ from (\ref{eq:62}), then
\[da_{33} + \pi_6  \equiv 0\ \ \ \ \ \ \ \ \   \mod\ \ \omega_1, \omega_2, \omega_3,\omega_4\, .\]
We can translate $a_{33} = 0$. Then
\[ \pi_6  \equiv 0\ \ \ \ \ \ \ \ \   \mod\ \ \omega_1, \omega_2, \omega_3,\omega_4\, .\]

\noindent Calculate $d^2\omega_3 = 0$ from (\ref{eq:62}), then
\[d(a_{23} - a_{14}) + \pi_5  \equiv 0\ \ \ \ \ \ \ \ \   \mod\ \ \omega_1, \omega_2, \omega_3,\omega_4\, .\]
We can translate $a_{23} - a_{14} = 0$. Then
\[ \pi_5  \equiv 0\ \ \ \ \ \ \ \ \   \mod\ \ \omega_1, \omega_2, \omega_3,\omega_4\, .\]

\noindent Calculate $d^2\omega_1 = 0$ from (\ref{eq:62}), then
\[a_{34} = a_{14}\]
and
\[d a_{12} + \pi_4  \equiv 0\ \ \ \ \ \ \ \ \   \mod\ \ \omega_1, \omega_2, \omega_3,\omega_4\, .\]
We can translate $a_{12} = 0$. Then
\[ \pi_4  \equiv 0\ \ \ \ \ \ \ \ \   \mod\ \ \omega_1, \omega_2, \omega_3,\omega_4\, .\]
Thus
\[\pi_4 \equiv\pi_5\equiv  \pi_6  \equiv 0\ \ \ \ \ \ \ \ \   \mod\ \ \omega_1, \omega_2, \omega_3,\omega_4\, .\]

Now we get a canonical coframing and the $G$-structure is reduced to an $e$-structure. The structure equation is
 \begin{equation}\label{eq:30}
 d\left[ {\begin{array}{c}
                                   \omega_1 \\
                                   \omega_2\\
                                   \omega_3\\
                                   \omega_4\\
                                    \end{array} } \right] = \left[ {\begin{array}{cccccc}     
                                    0 & a_{13} & a_{14} & 1 & 0& 0 \\
                                    a_{22} - a_{31} & a_{14} & a_{24} & 0& a_{14} & 1\\
                                    a_{42} - a_{51} & a_{11} + a_{43} & a_{44} + a_{21} & a_{53} & a_{54} + a_{22} & 0\\
                                    a_{52} - a_{61} & a_{53} & a_{54} + a_{31} & a_{63} & a_{64} + a_{32} & 0
                                     \end{array} } \right]                                                   
                                     \left[ {\begin{array}{c}
                                    \omega_1\wedge\omega_2\\
                                     \omega_1\wedge\omega_3\\
                                     \omega_1\wedge\omega_4\\
                                     \omega_2\wedge\omega_3\\
                                     \omega_2\wedge\omega_4\\
                                     \omega_3\wedge\omega_4        \end{array} } \right],
\end{equation}
where the nonzero terms of the right side of (\ref{eq:30}) represent intrinsic torsion of Lagrangian Engel structures. The coefficients of torsion terms are functional invariants of Lagrangian Engel structures. We have finished the analysis of the structure equation for the case $a_3 \neq 0$. 

\subsection{case $a_3 = 0$}
We can scale $a_4 = 1$.
Thus 
\[\pi_1 \equiv 0\ \ \ \ \ \ \ \ \   \mod\ \ \omega_1, \omega_2, \omega_3,\omega_4\, .\]

\noindent Calculate $d^2\omega_2 = 0$ from equation (\ref{eq:62}), then
\[da_0 + 2a_0 \pi_3  \equiv 0\ \ \ \ \ \ \ \ \   \mod\ \ \omega_1, \omega_2, \omega_3,\omega_4\, .\]
Since $a_0 \neq 0$, we can scale $a_0 = 1$. Then
\[ \pi_3  \equiv 0\ \ \ \ \ \ \ \ \   \mod\ \ \omega_1, \omega_2, \omega_3,\omega_4\, .\]

\noindent Calculate $d^2\omega_4 = 0$ from (\ref{eq:62}), then
\[da_{33} + \pi_6  \equiv 0\ \ \ \ \ \ \ \ \   \mod\ \ \omega_1, \omega_2, \omega_3,\omega_4\, .\]
We can translate $a_{33} = 0$. Then
\[ \pi_6  \equiv 0\ \ \ \ \ \ \ \ \   \mod\ \ \omega_1, \omega_2, \omega_3,\omega_4\, .\]

Also from $d^2\omega_4 = 0$,
\[d(a_{32} + a_{64}) + 2\pi_5 + a_{63}\pi_2  \equiv 0\ \ \ \ \ \ \ \ \   \mod\ \ \omega_1, \omega_2, \omega_3,\omega_4\, .\]
We can translate $a_{32} + a_{64}= 0$. Then
\[2\pi_5 + a_{63}\pi_2  \equiv 0\ \ \ \ \ \ \ \ \   \mod\ \ \omega_1, \omega_2, \omega_3,\omega_4\, .\]

\noindent Calculate $d^2\omega_2 = 0$ from (\ref{eq:62}), then
\[da_{34}  - \pi_5 + \pi_2  \equiv 0\ \ \ \ \ \ \ \ \   \mod\ \ \omega_1, \omega_2, \omega_3,\omega_4\, .\]
There are 2 cases: $a_{63} \neq -2$ or $a_{63} \equiv -2$.

\begin{enumerate}
\item case 1. $a_{63} \neq -2$
We can translate $a_{34} = 0$. Then
\[\pi_5 \equiv \pi_2  \equiv 0\ \ \ \ \ \ \ \ \   \mod\ \ \omega_1, \omega_2, \omega_3,\omega_4\, .\]
\item case 2. $a_{63} = -2$

\noindent From $d^2\omega_1 = 0$,
\[da_{14}  + (a_{13} - 1)\pi_2  \equiv 0\ \ \ \ \ \ \ \ \   \mod\ \ \omega_1, \omega_2, \omega_3,\omega_4\]
and
\[da_{12}  + (-a_{13} + 1)\pi_5  \equiv 0\ \ \ \ \ \ \ \ \   \mod\ \ \omega_1, \omega_2, \omega_3,\omega_4\, .\]
\end{enumerate}

In summary, as long as $a_{13} \neq 1$ or $a_{63} \neq -2$, 
\[\pi_5 \equiv \pi_2  \equiv 0\ \ \ \ \ \ \ \ \   \mod\ \ \omega_1, \omega_2, \omega_3,\omega_4\, .\]

\subsubsection{ $a_{13} \neq 1$ or $a_{63} \neq -2$}

\noindent From $d^2\omega_2 = 0$,
\[da_{24}  -\pi_4  \equiv 0\ \ \ \ \ \ \ \ \   \mod\ \ \omega_1, \omega_2, \omega_3,\omega_4\, .\]
We can translate $a_{24} = 0$. Then
\[\pi_4  \equiv 0\ \ \ \ \ \ \ \ \   \mod\ \ \omega_1, \omega_2, \omega_3,\omega_4\, .\]
So
\[\pi_1  \equiv  \pi_2  \equiv\pi_3  \equiv\pi_4  \equiv\pi_5  \equiv\pi_6  \equiv0\ \ \ \ \ \ \ \ \   \mod\ \ \omega_1, \omega_2, \omega_3,\omega_4\, .\]

Now we get an $e$-structure and a canonical coframing. In this case there are 2 different families of structure equations. 

%

\subsubsection{$a_{13} =1$ and $a_{63} = -2$}

In this case, we know that 
\[\pi_1  \equiv \pi_3  \equiv\pi_6  \equiv \pi_2 - \pi_5 \equiv 0\ \ \ \ \ \ \ \ \   \mod\ \ \omega_1, \omega_2, \omega_3,\omega_4\, .\]
From $d^2\omega_1 = 0$,
\[da_{12}   \equiv 0\ \ \ \ \ \ \ \ \   \mod\ \ \omega_1, \omega_2, \omega_3,\omega_4\]
and
\[da_{14}   \equiv 0\ \ \ \ \ \ \ \ \   \mod\ \ \omega_1, \omega_2, \omega_3,\omega_4\, .\]

I do not intend to finish the calculation of all invariants of Lagrangian Engel structures of this case. Since the goal is to classify compact quotients that support homogeneous Lagrangian Engel structures, I will prove that no compact quotients can support a homogeneous Lagrangian Engel structure of this case.

From the structure equation,
\[d(\omega_1\wedge\omega_2\wedge\omega_4) = -2\ \omega_1\wedge\omega_2\wedge\omega_3\wedge\omega_4\, .\]
By Stokes's Theorem, there is no compact quotient that supports a homogeneous Lagrangian Engel structure when $a_{13} =1$ and $a_{63} = -2$. In the following classification of compact homogeneous Lagrangian Engel structures, we will not consider this case any more.

Now we will classify homogeneous Lagrangian Engel structures. Assume all the coefficients in the structure equations are constants. By taking exterior derivative of the structure equation and setting all coefficients to zero, we can get quadratic equations of the constants. Via MAPLE, we can solve all the equations. The structure equations of homogeneous Lagrangian Engel structures are listed in the statement of the theorem.

For the case that $a_{13} =1$ and $a_{63} = -2$, it remains to determine whether there exist homogeneous Lagrangian Engel structures.
\end{proof}

\section{Classification of Compact Homogeneous Lagrangian Engel Structures}
 \begin{theorem}[Classification of Compact Homogeneous Lagrangian Engel Structures]
There is only a 1-parameter family of compact homogeneous Lagrangian Engel structures. There exists a canonical coframing $(\omega_1, \omega_2, \omega_3, \omega_4)$ such that
 \[d\left[ {\begin{array}{c}
                                   \omega_1 \\
                                   \omega_2\\
                                   \omega_3\\
                                   \omega_4\\
                                    \end{array} } \right] =\left[ {\begin{array}{c}
                                   \omega_2\wedge\omega_3 \\
                                 \omega_3\wedge\omega_4\\
                                 0\\
                                b\omega_2\wedge\omega_3\\
                                    \end{array} } \right]\]
                               
\noindent where $b\in\mathbb{R}$ is a constant.
\end{theorem}

\begin{proof}

We will prove this theorem by analyzing each homogeneous case in Theorem \ref{th:hle} and determining whether there exists a compact quotient that can support the corresponding homogeneous Lagrangian Engel structure of one particular case.
\subsection{Analysis of Case 1}
The structure equation is
 \begin{equation}\label{eq:67}
 d\left[ {\begin{array}{c}
                                   \omega_1 \\
                                   \omega_2\\
                                   \omega_3\\
                                   \omega_4\\
                                    \end{array} } \right] =\left[ {\begin{array}{c}
                                   \omega_2\wedge\omega_3 + a\omega_1\wedge\omega_3 \\
                                 \omega_3\wedge\omega_4\\
                                 0\\
                                b\omega_2\wedge\omega_3\\
                                    \end{array} } \right]
\end{equation}
where $a$ and $b$ are constants.
From the structure equation (\ref{eq:67}),
\[d(\omega_1\wedge\omega_2\wedge\omega_4) = -a\ \omega_1\wedge\omega_2\wedge\omega_3\wedge\omega_4\, .\]
Thus if $a\neq 0$, there is no compact quotient that can support a homogeneous Lagrangian Engel structure of case 1.

In the following, we only consider $a = 0$.
Since $d\omega_3 = 0$ and $d(\omega_4 - b\omega_1) = 0$, there exist functions $x$ and $y$ such that
\begin{align*}
\omega_3 &= dy,\\
\omega_4 - b\omega_1 &= dx.
\end{align*}
Define $\tilde\omega_2 = \omega_2 + xdy$. Then
 \begin{equation}\label{eq:102}
 d\left[ {\begin{array}{c}
                                   \omega_1 \\
                                   \tilde\omega_2\\
                                    \end{array} } \right] =\left[ {\begin{array}{c}
                                   \tilde\omega_2\wedge dy \\
                                 b\ dy\wedge\omega_1\\
                                    \end{array} } \right].
\end{equation}
\begin{proposition}
If $a = 0$, there exists a compact quotient that supports a homogeneous Lagrangian Engel structure of Case1 for any $b$.
\end{proposition}    
 We prove this proposition by considering different values for $b$.                               
\subsubsection{$b = 0$}

 \[d\left[ {\begin{array}{c}
                                   \omega_1 \\
                                   \omega_2\\
                                    \end{array} } \right] =\left[ {\begin{array}{c}
                                  \omega_2\wedge dy \\
                                dy\wedge dx\\
                                    \end{array} } \right].\]
Thus there exist functions $u$ and $v$ such that
\begin{align*}
\omega_1 &= udy + dv\, ,\\
\omega_2 &= -xdy + du\, .
\end{align*}
Since 
\[\omega_1\wedge\omega_2\wedge\omega_3\wedge\omega_4 = dv\wedge du\wedge dy\wedge dx\, ,\]
so $x,y,u,v$ can be a local coordinate system for the homogeneous manifold.

Let 
\begin{equation}\label{eq:74}
\omega = \left[ {\begin{array}{cccc}
                                  0 & \omega_3 & -\omega_2 & 2\omega_1 \\
                                  0 & 0 & \omega_4 & \omega_2\\
                                  0 & 0 & 0 &\omega_3\\
                                  0 & 0 & 0 & 0\\
                                    \end{array} } \right]
\end{equation}
be a matrix-valued 1-form. Then from (\ref{eq:67}), we have
\begin{equation}\label{eq:75}
d\omega = -\omega\wedge\omega
\end{equation}
Thus $\omega$ is a left-invariant form of a Lie group $G$. The connected and simply-connected Lie group corresponding to the left-invariant form in (\ref{eq:75}) is isomorphic to 
\begin{equation}
G = \left\{
\left.\left[ {\begin{array}{cccc}
                                  1 & f & fe - c & d \\
                                  0 & 1 & 2e & fe + c \\
                                  0 & 0 & 1 & f\\
                                  0 & 0 & 0 & 1\\
                                    \end{array} } \right]  \right\vert \text{ where } f, e, c, d\in \mathbb{R}
\right\}
\end{equation}

Note $G$ is a nilpotent Lie group. In \cite{MR0507234}, there is a theorem:
\begin{theorem}\label{th:1}
A simply-connected nilpotent Lie group $G$ admits a lattice if and only if there exists a basis $(X_1, X_2,\cdots, X_n)$ of the Lie algebra $\mathfrak{g}$ of $G$ such that the structure constants $C^k_{ij}$ arising in the brackets
\begin{equation}\label{eq:76}
[X_i, X_j] = \sum_k C^k_{ij}X_k
\end{equation}
are rational numbers.
\end{theorem}
By the structure (\ref{eq:75}) and Theorem \ref{th:1}, there exists a co-compact lattice for the group $G$, and thus there exists a compact quotient that can support a homogeneous Lagrangian Engel structure. We will find an explicit co-compact lattice in this case. Take a discrete subgroup of Lie group $G$
\begin{equation}
\Gamma = \left\{
\left.\left[ {\begin{array}{cccc}
                                  1 & f & fe - c & d \\
                                  0 & 1 & 2e & fe + c \\
                                  0 & 0 & 1 & f\\
                                  0 & 0 & 0 & 1\\
                                    \end{array} } \right] \right\vert \text{ where } c, d, e, f\in \mathbb{Z}\right\} .
\end{equation}
It is easy to verify that $\Gamma$ is a subgroup of $G$ and that $M = G/\Gamma$
is compact. So if $b = 0$, there exists a compact quotient, that supports a homogeneous Lagrangian Engel structure.

\subsubsection{$b < 0$}
Set $b = -\beta^2$, where $\beta > 0$. Then by (\ref{eq:102}), we get
\[d(\beta\omega_1 + \tilde\omega_2) = \beta (\beta \omega_1 + \tilde\omega_2)\wedge dy\]
and
\[d(-\beta \omega_1 + \tilde\omega_2) = -\beta (-\beta \omega_1 + \tilde\omega_2)\wedge dy\, .\]
So there exist functions $u$ and $v$ such that $\beta \omega_1 + \tilde\omega_2 = e^{-\beta y}du $ and $-\beta \omega_1 + \tilde\omega_2 = e^{\beta y}dv$. Thus 
\[\tilde\omega_2 = \frac{e^{-\beta y}du + e^{\beta y}dv}{2}\]
and
\[\omega_1 = \frac{e^{-\beta y}du - e^{\beta y}dv}{2\beta}.\]


Now we take a new coframing. After scaling $\omega_4\rightarrow \beta\omega_4$,  $\omega_3\rightarrow \frac{1}{\beta}\omega_3$ and  $\omega_1\rightarrow \frac{1}{\beta}\omega_1$, then the structure equation is transformed to
 \begin{equation}
 d\left[ {\begin{array}{c}
                                   \omega_1 \\
                                   \omega_2\\
                                   \omega_3\\
                                   \omega_4\\
                                    \end{array} } \right] =\left[ {\begin{array}{c}
                                   \omega_2\wedge\omega_3  \\
                                 \omega_3\wedge\omega_4\\
                                 0\\
                                -\omega_2\wedge\omega_3\\
                                    \end{array} } \right].
\end{equation}
Define $\omega_0 = \omega_1 + \omega_4$, $\tilde\omega_2 = \omega_2 + \omega_4$ and $\tilde\omega_4 = \omega_2 - \omega_4$. Then $(\omega_0, \tilde\omega_2, \omega_3, \tilde\omega_4)$ is a new coframing. In this new coframing, after dropping tildes, the structure equation is
 \begin{equation}\label{eq:71}
 d\left[ {\begin{array}{c}
                                   \omega_0 \\
                                   \omega_2\\
                                   \omega_3\\
                                   \omega_4\\
                                    \end{array} } \right] =\left[ {\begin{array}{c}
                                  0\\
                                 \omega_3\wedge\omega_2\\
                                 0\\
                                -\omega_3\wedge\omega_4\\
                                    \end{array} } \right].
\end{equation}

Let 
\begin{equation}\label{eq:69}
\omega = \left[ {\begin{array}{cccc}
                                  \omega_0 & 0 & 0 & 0 \\
                                  0 & -\omega_3 & 0 & \omega_2\\
                                  0 & 0 & \omega_3 &\omega_4\\
                                  0 & 0 & 0 & 0\\
                                    \end{array} } \right].
\end{equation}
be a matrix-valued 1-form. Then from (\ref{eq:71})
\begin{equation}
d\omega = -\omega\wedge\omega.
\end{equation}
Thus $\omega$ is a Maurer-Cartan form of a Lie group $G$. The connected and simply-connected Lie group corresponding to the Maurer-Cartan form in (\ref{eq:69}) is isomorphic to 
\begin{equation}
G = \left\{
\left.\left[ {\begin{array}{cccc}
                                  c & 0 & 0 & 0 \\
                                  0 & t^{-1} & 0 &  r \\
                                  0 & 0 & t & s\\
                                  0 & 0 & 0 & 1\\
                                    \end{array} } \right] \right\vert \text{ where }  r, s\in \mathbb{R}
 \text{ and } c > 0 \text{ and } t > 0\right\}.
\end{equation}
\begin{theorem}
There exists a co-compact lattice of $G$.
\end{theorem}
\begin{proof}
Let $(X_1, X_2, X_3, X_4)$ be the left-invariant vectors dual to the left-invariant 1-forms $(\omega_1, \omega_3, -\omega_2, \omega_0)$, respectively. Then the nontrivial brackets are
\begin{align*}
[X_1, X_3] &= X_1\, ,\\
[X_2, X_3] &= - X_2\, .
\end{align*} 
By the classification results of \cite{MR3480018}, there exists a co-compact lattice.
\end{proof}

We will give an explicit way to construct a lattice. Consider a subgroup $H\subset G$, where
\begin{equation}
H = \left\{
\left.\left[ {\begin{array}{ccc}
                              
                                   t^{-1} & 0 &  r \\
                                  0 & t & s\\
                                  0 & 0 & 1\\
                                    \end{array} } \right] \right\vert \text{ where }  r, s\in \mathbb{R}
 \text{ and }  t > 0\right\}
\end{equation}
and the inclusion map of $H$ to $G$ is
\begin{equation}
\left[ {\begin{array}{ccc}                              
                                   t^{-1} & 0 &  r \\
                                  0 & t & s\\
                                  0 & 0 & 1\\
                                    \end{array} } \right] 
                                    \longrightarrow
                                    \left[ {\begin{array}{cccc}
                                  1 & 0 & 0 & 0 \\
                                  0 & t^{-1} & 0 &  r \\
                                  0 & 0 & t & s\\
                                  0 & 0 & 0 & 1\\
                                    \end{array} } \right].
\end{equation}
Then $G \cong \mathbb{R}\times H$ as a group. Let $N\subset H$ be the subgroup
\begin{equation}
N = \left\{
\left.\left[ {\begin{array}{ccc}
                              
                                  1 & 0 &  r \\
                                  0 & 1 & s\\
                                  0 & 0 & 1\\
                                    \end{array} } \right] \right\vert \text{ where }  r, s\in \mathbb{R}\right\}.
\end{equation}
\begin{lemma}
$N$ is a normal subgroup of $H$.
\end{lemma}
\begin{proof}
Let $h = \left[ {\begin{array}{ccc}                              
                                   t^{-1} & 0 &  x \\
                                  0 & t & y\\
                                  0 & 0 & 1\\
                                    \end{array} } \right] $ and $n = \left[ {\begin{array}{ccc}
                              
                                  1 & 0 &  r \\
                                  0 & 1 & s\\
                                  0 & 0 & 1\\
                                    \end{array} } \right]$ be any elements of $H$ and $N$, respectively. Then
\begin{equation}
h^{-1}nh = \left[ {\begin{array}{ccc}
                              
                                  1 & 0 &  t\cdot r \\
                                  0 & 1 & t^{-1}\cdot s\\
                                  0 & 0 & 1\\
                                    \end{array} } \right]\in N.
\end{equation}
Hence, $N$ is a normal subgroup of $H$.
\end{proof}
Thus
\[H/N \cong  \left\{
\left.\left[ {\begin{array}{cc}                             
                                  t^{-1} & 0 \\
                                  0 & t \\
                                    \end{array} } \right] \right\vert \text{ where } t > 0\right\}
                                    \]
 is a quotient group.                                   
Let $L_1 = \langle \vec{v}_1, \vec{v}_2  \rangle\subset \mathbb{R}^2$ be a lattice of the normal subgroup $N$, to be determined later. We need to find a lattice $L_2$ of $H/N$ such that the lattice of the group $H$ is
\[L  = \left\{
\left.\left[ {\begin{array}{cc}
                              
                                  \gamma & \vec{v} \\
                                  0 & 1\\
                                    \end{array} } \right] \right\vert \text{ where } \gamma\in L_2, \vec{v}\in L_1\right\}.
                                    \]

From the multiplication rule of the group $H$, this is equivalent to $\gamma \vec{v}\in L_1$ for any $\gamma \in L_2$ and any $\vec{v} =  \left[ {\begin{array}{c}v_1 \\ v_2 \end{array} } \right] \in L_1$. Hence we need to find $a_1, a_2, a_3, a_4\in \mathbb{Z}$  and $c > 0$ and $c\neq 1$ such that
\begin{align}\label{eq:70}
\gamma_c \vec{v}_1 &= a_1 \vec{v}_1 + a_2 \vec{v}_2\, ,\nonumber\\                             
\gamma_c \vec{v}_2 &= a_3 \vec{v}_1 + a_4 \vec{v}_2\, ,                                   
\end{align}
where $\gamma_c$ is the linear transform with transformation matrix $\left[ {\begin{array}{cc}
                              
                                  c^{-1} & 0 \\
                                  0 & c\\
                                    \end{array} } \right]$. Since $\langle\gamma_c \vec{v}_1, \gamma_c \vec{v}_2\rangle$ form a new basis for the lattice $L_1$, then  $\left[ {\begin{array}{cc}
                              
                                  a_1 & a_2 \\
                                  a_3 & a_4\\
                                    \end{array} } \right] \in SL_2(\mathbb{Z})$.
                                    
Thus (\ref{eq:70}) is equivalent to
\begin{equation}
\left[ {\begin{array}{cc}
                              
                                  c^{-1} & 0 \\
                                  0 & c\\
                                    \end{array} } \right] = (\vec{v}_1, \vec{v}_2) \left[ {\begin{array}{cc}
                              
                                  a_1 & a_3 \\
                                  a_2 & a_4\\
                                    \end{array} } \right] (\vec{v}_1, \vec{v}_2)^{-1}
\end{equation}

So we can choose any matrix $S = \left[ {\begin{array}{cc}
                              
                                  A & B \\
                                  C & D\\
                                    \end{array} } \right] \in SL_2(\mathbb{Z})$ such that $(A + D)^2 - 4 > 0$, then we have two real eigenvalues $\lambda_1 > \lambda_2 > 0$. We can set $c = \lambda_1$.  If $(\vec{v}, \vec{w})$ are eigenvectors of $S$ with eigenvalues $(\lambda_1, \lambda_2)$, then we can set $(\vec{v}_1, \vec{v}_2) \propto (\vec{v}, \vec{w})^{-1}$.

\begin{example}                                                                        
Take $S = \left[ {\begin{array}{cc}                            
                                  2 & 1 \\
                                  1 & 1\\
                                    \end{array} } \right]\in SL_2(\mathbb{Z})$, then $c = \frac{3 + \sqrt{5}}{2}$. We can take $\vec{v} = \left[ {\begin{array}{c}
                              
                                  1  \\
                                  \frac{-1 - \sqrt{5}}{2} \\
                                    \end{array} } \right]$ and $\vec{w} = \left[ {\begin{array}{c}
                              
                                  1  \\
                                  \frac{-1 + \sqrt{5}}{2} \\
                                    \end{array} } \right]$.  
Thus
\[(\vec{v}_1, \vec{v}_2) \propto (\vec{v}, \vec{w})^{-1} =   \frac{1}{\sqrt{5}}    \left[ {\begin{array}{cc}
                              
                                   \frac{-1 + \sqrt{5}}{2}  & -1  \\
                                  \frac{1 + \sqrt{5}}{2} & 1 \\
                                    \end{array} } \right].    \]                                   
                                    
Then the lattice of the Lie group $H$ can be
\begin{equation}
L =\left\{\left. \left[ {\begin{array}{ccc}                            
                                   \left(\frac{3 + \sqrt{5}}{2}\right)^{-m_0}& 0 & m_1\left( \frac{-1 + \sqrt{5}}{2}\right) - m_2   \\
                                  0 &  \left(\frac{3 + \sqrt{5}}{2}\right)^{m_0} & m_1 \left(\frac{1 + \sqrt{5}}{2} \right)+ m_2  \\
                                  0 & 0 & 1
                                    \end{array} } \right]
                                    \right\vert
                                    \text{ where } m_0, m_1, m_2\in\mathbb{Z}
                                    \right\}.
\end{equation}

\end{example}
Thus for $b < 0$, there exists a lattice $\Gamma$ such that $G/\Gamma\cong S^1\times H/L$ is compact.

\begin{remark}
In our analysis of the existence of a lattice for $H$, we know that the different lattices correspond to

\begin{enumerate}
\item scaling or change of basis for eigenvectors of a matrix in $SL_2(\mathbb{Z})$
\item different matrices in $SL_2(\mathbb{Z})$ such that the absolute value of trace is greater than 2
\end{enumerate}
\end{remark}

\subsubsection{$b > 0$}
Set $b = \beta^2$, where $\beta > 0$. Then by (\ref{eq:102}), there exist functions $u$ and $v$ such that
\[i\beta\omega_1 + \tilde\omega_2 = e^{-i\beta y} d(u + iv)\]
Take the real and imaginary part of the 1-form, we can get
\[\omega_1 = \frac{1}{\beta}(\cos{\beta y}\, dv - \sin{\beta y}\, du)\]
and
\[\tilde\omega_2 = \cos{\beta y}\, du + \sin{\beta y}\, dv\, .\]
\begin{theorem}
There exists a co-compact lattice.
\end{theorem}
\begin{proof}
Let $(e_1, e_2, e_3, e_4)$ be the left-invariant vectors dual to the left-invariant 1-forms $(\beta\omega_1, \omega_2, \beta\omega_3, \frac{1}{\beta}\omega_4)$, respectively. Then the nontrivial brackets are
\begin{align*}
[e_2, e_3] &= e_1 + e_4\, ,\\
[e_3, e_4] &= e_2\, .
\end{align*} 
Define $X_1 = e_1 + e_4, X_2 = e_2, X_3 = e_3, X_4 = e_1$, then the nontrivivial brackets are
\begin{align*}
[X_1, X_3] &= -X_2\, ,\\
[X_2, X_3] &= X_1\, .
\end{align*} 
By the classification results of \cite{MR3480018}, there exists a co-compact lattice.
\end{proof}
So if $b>0$, we get a compact quotient that supports a homogeneous Lagrangian Engel structure.

In summary, in case 1 we can get a compact quotient if and only if $a = 0$.

\subsection{Analysis of Case 2}
The structure equation is
 \[d\left[ {\begin{array}{c}
                                   \omega_1 \\
                                   \omega_2\\
                                   \omega_3\\
                                   \omega_4\\
                                    \end{array} } \right] =\left[ {\begin{array}{c}
                                   \omega_2\wedge\omega_3 + a\omega_1\wedge\omega_3 + b\omega_1\wedge\omega_4\\
                                b\omega_1\wedge\omega_3 + \omega_3\wedge\omega_4 + b\omega_2\wedge\omega_4\\
                                 0\\
                               0\\
                                    \end{array} } \right],\]
where $a$ and $b$ are constants. We can assume $b \neq 0$, otherwise, this is a special case of case 1.

From the structure equation,
\[d(\omega_1\wedge\omega_2\wedge\omega_3) = 2b\, \omega_1\wedge\omega_2\wedge\omega_3\wedge\omega_4. \]
Since $b \neq 0$, there does not exist a compact quotient that supports a homogeneous Lagrangian Engel structure in case 2.

\subsection{Analysis of Case 3}
The structure equation is
 \[d\left[ {\begin{array}{c}
                                   \omega_1 \\
                                   \omega_2\\
                                   \omega_3\\
                                   \omega_4\\
                                    \end{array} } \right] =\left[ {\begin{array}{c}
                                   \omega_2\wedge\omega_3 + a\omega_1\wedge\omega_3 - \frac{1}{4} a^2\omega_1\wedge\omega_4\\
                                -\frac{1}{4}a^2\omega_1\wedge\omega_3 + \omega_3\wedge\omega_4  -\frac{1}{4}a^2\omega_2\wedge\omega_4\\
                                 \frac{1}{2}a^2b(\omega_1\wedge\omega_3 - \omega_2\wedge\omega_4) + ab\omega_2\wedge\omega_3 - \frac{1}{4}a^3b\omega_1\wedge\omega_4\\
                               ab(\omega_1\wedge\omega_3 - \omega_2\wedge\omega_4) + 2b \omega_2\wedge\omega_3 - \frac{1}{2}a^2b\omega_1\wedge\omega_4\\
                                    \end{array} } \right],\]
where $a$ and $b$ are constants. Since
\[d(\omega_1\wedge\omega_2\wedge\omega_3) = -\frac{a^2}{2}\, \omega_1\wedge\omega_2\wedge\omega_3\wedge\omega_4,\]
there does not exist a compact quotient that supports a homogeneous Lagrangian Engel structure of case 3 if $a \neq 0$. In the following, we assume $a = 0$ and the structure equation is
 \[d\left[ {\begin{array}{c}
                                   \omega_1 \\
                                   \omega_2\\
                                   \omega_3\\
                                   \omega_4\\
                                    \end{array} } \right] =\left[ {\begin{array}{c}
                                   \omega_2\wedge\omega_3 \\
                                 \omega_3\wedge\omega_4\\
                                 0\\
                                2b \omega_2\wedge\omega_3\\
                                    \end{array} } \right].\]
This is a special case of case 1 with $a = 0$. There exists a compact quotient for any $b$.

\subsection{Analysis of Case 4}
The structure equation is
 \[d\left[ {\begin{array}{c}
                                   \omega_1 \\
                                   \omega_2\\
                                   \omega_3\\
                                   \omega_4\\
                                    \end{array} } \right] =\left[ {\begin{array}{c}
                                   \omega_2\wedge\omega_3 + b\omega_1\wedge\omega_3 + a \omega_1\wedge\omega_4\\
                               (a^2 + \frac{ab^2}{4})\omega_1\wedge\omega_2 + a\omega_1\wedge\omega_3 + \omega_3\wedge\omega_4 + a\omega_2\omega_4 \\
                                (2a^3 + \frac{1}{2}a^2b^2)\omega_{12} + \frac{ab^2}{2}\omega_{13} + b\omega_{23} + a^2b\omega_{14} + 2a^2\omega_{24}\\
                               ab(-a - \frac{1}{4})\omega_{12} + ab\omega_{13} + (a - \frac{1}{4})\omega_{23} + (a^2 - \frac{ab^2}{4})\omega_{14} -ab\omega_{24}\\
                                    \end{array} } \right],\]
where $a$ and $b$ are constants.

By the structure equation
\begin{align*}
d(\omega_1\wedge\omega_2\wedge\omega_3) &= 2a\, \omega_1\wedge\omega_2\wedge\omega_3\wedge\omega_4\, ,\\
d(\omega_1\wedge\omega_2\wedge\omega_4) &= -b\, \omega_1\wedge\omega_2\wedge\omega_3\wedge\omega_4
\end{align*}
If $a \neq 0$ or $b \neq 0$, there does not exist a compact quotient that supports a homogeneous Lagrangian Engel structure in case 4. If $a =0$ and $b = 0$, this is a special case of case 1 with compact quotients.

\subsection{Analysis of Case 5}
The structure equation is
 \[d\left[ {\begin{array}{c}
                                   \omega_1 \\
                                   \omega_2\\
                                   \omega_3\\
                                   \omega_4\\
                                    \end{array} } \right] =\left[ {\begin{array}{c}
                                   \omega_1\wedge\omega_3 + \omega_2\wedge\omega_4\\
                                  a\omega_1\wedge\omega_2 + \omega_3\wedge\omega_4\\
                                  a( \omega_1\wedge\omega_3 + \omega_2\wedge\omega_4)\\
                                -a  \omega_1\wedge\omega_2 - \omega_3\wedge\omega_4\\
                                    \end{array} } \right],\]
where $a$ is a constant.

By the structure equation
\[d(\omega_1\wedge\omega_2\wedge\omega_4) = -2\, \omega_1\wedge\omega_2\wedge\omega_3\wedge\omega_4.\]
Thus there does not exist a compact quotient that supports a homogeneous Lagrangian Engel structure in case 5.

\subsection{Analysis of Case 6}
The structure equation is
 \[d\left[ {\begin{array}{c}
                                   \omega_1 \\
                                   \omega_2\\
                                   \omega_3\\
                                   \omega_4\\
                                    \end{array} } \right] =\left[ {\begin{array}{c}
                                   \omega_1\wedge\omega_3 + \omega_2\wedge\omega_4\\
                                  \omega_3\wedge\omega_4\\
                                  0\\
                                 a\omega_2\wedge\omega_3 - \omega_3\wedge\omega_4\\
                                    \end{array} } \right],\]
where $a$ is a constant.

By the structure equation
\[d(\omega_1\wedge\omega_2\wedge\omega_4) = -2\, \omega_1\wedge\omega_2\wedge\omega_3\wedge\omega_4.\]
Thus there does not exist a compact quotient that supports a homogeneous Lagrangian Engel structure in case 6.

\end{proof}

\end{document}